\documentclass[letterpaper, 10 pt, conference]{ieeeconf}  

\IEEEoverridecommandlockouts                              

\usepackage{subcaption}
\usepackage{amsmath,amssymb}
\usepackage{graphicx,epstopdf}
\usepackage{cite}
\usepackage{color}
\usepackage{amsfonts}
\usepackage{stfloats}
\usepackage{bm}
\usepackage[symbol]{footmisc}
\usepackage{algorithmic}

\newcommand{\R}{\mathbb{R}}

\newcommand{\Kzcbf}{K_{\mathrm{zcbf}}}

\newcommand{\etal}{\textit{et al.}}
\newcommand{\la}{\lambda}
\newcommand*\diff{\mathop{}\!\mathrm{d}}

\allowdisplaybreaks[4]

\newif\ifdraft
\drafttrue

\newtheorem{definition}{\bfseries Definition}
\newtheorem{example}{\bfseries Example}
\newtheorem{assumption}{\it Assumption}
\newtheorem{theorem}{\bfseries Theorem}

\newtheorem{remark}{\bfseries Remark}

\newtheorem{problem}{\bfseries Problem}

\title{\LARGE \bf
Observer-based Control Barrier Functions for Safety Critical Systems
}

\author{}

\author{Yujie Wang and Xiangru Xu\thanks{Y. Wang and X. Xu are with the Department of Mechanical Engineering, University of Wisconsin-Madison,
        Madison, WI 53706, USA. Email:
\{ywang2835,xiangru.xu\}@wisc.edu.}}

\begin{document}

\maketitle



\begin{abstract}
This paper considers the safety-critical control design problem with output measurements. An observer-based safety control framework that integrates the estimation error quantified observer and the control barrier function (CBF) approach is proposed. The function approximation technique is employed to approximate the uncertainties introduced by the state estimation error, and an adaptive CBF approach is proposed to design the safe controller which is obtained by solving a convex quadratic program (QP). Theoretical results for CBFs with a relative degree 1 and a higher relative degree are given individually. The effectiveness of the proposed control approach is demonstrated by two numerical examples.
\end{abstract}






\section{Introduction}
\label{sec:introduction}
The safety of autonomous systems has drawn increasing attention in the past decades with various control techniques been developed \cite{aswani2013provably,mitchell2005time,althoff2014online,dong2017safety}. Control barrier function (CBF) has become a powerful tool for achieving safety in the form of set invariance \cite{ames2016control,Xu2015ADHS}, and has been applied to many scenarios  including adaptive cruise control \cite{xu2017correctness}, biped robots \cite{hsu2015control,nguyen20163d}, and UAVs \cite{wang2017safe}. Almost all the existing results using CBFs rely on accurate state information, which is hard to obtain in real applications. For example, in the absence of velocity sensors, the angular velocities of robot manipulators  cannot be obtained exactly;  even when the manipulator is equipped with velocity sensors, the velocity signals are usually contaminated by noise.

Various safe control methods have been developed in the absence of accurate state measurements \cite{dean2020guaranteeing,abu2021feedback,poonawala2021training,takano2018robust}. In \cite{dean2020guaranteeing}, a function mapping from outputs to states is learned via supervised learning techniques, and the controller is designed under the assumption that for any given output value, all possible state estimation error is bounded by a known constant.
Khalaf \etal \cite{abu2021feedback} proposed a controller synthesis approach involving feedback from pixels, which does not require feature extraction, object detection, or state estimation.
Poonawala \etal \cite{poonawala2021training} developed a method that trains classifiers for sensor-based control problems, bypassing the state estimation step.
Takano and Yamakita \cite{takano2018robust}  proposed a  QP-based controller with an unscented Kalman filter which
is capable of attenuating the effects of state disturbances and
measurement noises.
Nevertheless, assumptions in the aforementioned works are difficult to satisfy and limit their applicability in safety-critical control.

This paper considers the safety-critical control design problem with output measurements. The main contribution of this work lies in a novel observer-based CBF framework that integrates the estimation error quantified (EEQ) observer and the CBF approach, which generates provably safe controllers under mild conditions.  The residual terms introduced by the state estimation error are approximated by the function approximation technique (FAT) via a given set of basis functions with unknown weights. An adaptive CBF method is proposed to guarantee the safety of the controlled system, where the unknown weights are estimated by adaptive laws. The EEQ observer considered in this work can be not only the traditional asymptotic observer but also interval observers \cite{efimov2013control} and neural-network-based observers \cite{marchi2021safety,elkenawy2020diagonal}.

The rest of this paper is organized as follows. Section~\ref{sec:review} reviews some preliminaries about FAT,  CBF and estimation error quantified observers, and presents the problem statement. Section~\ref{sec:control} provides the main result of this paper for CBFs with a relative degree 1 and a higher relative degree, respectively. The effectiveness of the proposed control design method is demonstrated by simulations in Section~\ref{sec:simulation}. Finally, conclusions are drawn in Section~\ref{sec:concl}.

\section{Preliminaries \& Problem Statement}
\label{sec:review}
\subsection{Function Approximation Technique}\label{sec:FAT}

FAT is an effective tool for approximating time-varying unknown nonlinear functions \cite{huang2001sliding}.
Its basic idea is to express an unknown nonlinear function as the  combination of a set of given basis functions.  There are a lot of examples of FAT, such as the generalized Fourier series \cite{chen2005adaptive}, neural networks \cite{gong2001neural}, and differential equations \cite{izadbakhsh2018robust}.
In this paper, the basis functions $\varphi_i(t)$ are selected as trigonometric functions (the basis functions of Fourier series), which have been used in numerous papers \cite{bai2021adaptive,wang2016distributed,zirkohi2018direct,izadbakhsh2019fat}. Specifically,
$\varphi_i$ is defined as
\begin{equation}\label{dii}
  \varphi_{i}(t)=
  \begin{cases}
   \ 1, & i=0, \\
   \ \cos k \omega t, & i=2k-1, \\
   \ \sin k \omega t, & i=2k.
  \end{cases}
\end{equation}
Note that $\varphi_i(t)$ satisfies the orthonormal property, i.e., $\int_{t_1}^{t_2}\varphi_i(t)\varphi_j(t)dt=\delta_{ij}$, where $\delta_{ij}$ is the Kronecker delta. There are other  options for $\varphi_i$, including Bernstein polynomials, Legendre polynomials, and Chebyshev polynomials.

An arbitrary square integrable function $f(t): \R\rightarrow \R$ can be approximated by a generalized Fourier series in the interval $[t_1, t_2]$ as $f(t)=\sum_{i=1}^N \theta_i\varphi_i(t)+\epsilon_N(t)$ \cite{chen2005adaptive},
where $N$ is a given integer, $\theta_i$ is the corresponding coefficient, $\epsilon_N(t)$ is the truncation error satisfying $\lim_{N\to\infty}\int_{t_1}^{t_2}|\epsilon_N(t)|^2 \text{d} t=0$, and $\varphi_i(t), i=1,2,\cdots,N$ are defined in \eqref{dii}.
Since $\epsilon_N(t)$ vanishes as $N \to \infty$, $f(t)$ can be expressed as
\begin{equation*}
f(t)=\sum_{i=1}^\infty \theta_i \varphi_i(t).
\end{equation*}
Since $f(t)$ is an unknown function, $\theta_i$ cannot be directly calculated from the integral of $f(t)$. Hence, the majority of FAT-based controllers employ adaptive control techniques to estimate $\theta_i$ online, such that the estimation of $f(t)$ can be expressed as
\begin{equation}\label{hatf}
\hat{f}(t)=\sum_{i=1}^{N} \hat{\theta}_i\varphi_i(t)
\end{equation}
where $\hat{\theta}_i$ is the estimation of $\theta_i$ and governed by corresponding adaptive laws \cite{bai2021adaptive,huang2004adaptive,izadbakhsh2019fat}. For a vector function $f(t): \R\rightarrow \R^n$, the  approximation above holds for a set of vector parameters $\theta_i,\hat \theta_i\in\R^n$.

\subsection{Control Barrier Function}
Consider the following control affine system with output measurement:
\begin{IEEEeqnarray}{rCl}
\dot{x} &= & f(x) + g(x) u,\label{eqnsys}\\
y &=& l(x),\label{output}
\end{IEEEeqnarray}
where $x\in\mathbb{R}^n$ is the state, $u\in U\subset\mathbb{R}^m$ is the control input, $l: \mathbb{R}^n\to\mathbb{R}^k$ is the output measurement, $f: \mathbb{R}^n\to\mathbb{R}^n$ and $g:\mathbb{R}^n\to\mathbb{R}^{n\times m}$ are locally Lipchitz continuous functions.
A set $\mathcal{S}$ is called forward controlled invariant with respect to system \eqref{eqnsys} if for every $x_0 \in \mathcal{S}$, there exists a control signal $u(t)$ such that $x(t;t_0,x_0) \in \mathcal{S}$ for all $t\geq t_0$, where $x(t;t_0,x_0)$ denotes the solution of \eqref{eqnsys} at time $t$ with initial condition $x_0\in\mathbb{R}^n$ at time $t_0$. To simplify the discussion, we will use the same definition as above for the controlled invariance of  time-varying systems, which is slightly different from the definition given in  \cite{blanchini2008set}.

Consider control system \eqref{eqnsys} and a set $\mathcal{C} \subset \R^n$ defined by
\begin{equation}\label{setc}
    \mathcal{C} = \{ x \in \R^n : h(x) \geq 0\}
\end{equation}
for a continuously differentiable function $h: \R^n \to \R$ that has a relative degree 1. The function $h$ is called a (zeroing) CBF
if  there exists a constant $\gamma>0$ such that
\begin{align}\label{ineqZCBF}
& \sup_{u \in U}  \left[ L_f h(x) + L_g h(x) u + \gamma h(x)\right] \geq 0
\end{align}
where $L_fh(x)=\frac{\partial h}{\partial x}f(x)$ and $L_gh(x)=\frac{\partial h}{\partial x}g(x)$ are Lie derivatives \cite{Xu2015ADHS}.
Given a CBF $h$, the set of all control values that satisfy \eqref{ineqZCBF} for all $x\in\R^n$ is defined as:
It was proven in \cite{Xu2015ADHS} that any Lipschitz continuous controller $u(x) \in \Kzcbf(x)$ for every $x\in\R^n$ will guarantee the forward invariance of $\mathcal{C}$. The provably safe control law is obtained by solving an online quadratic program (QP) that includes the control barrier condition as its constraint.

A $C^r$ function $h(x): \R^n \to \R$ with a relative degree $r$ where $r\geq 2$ is called a (zeroing) CBF if  there exists a column vector ${\bf a}\in\R^r$ such that $\forall x\in\R^n$,
\begin{align}\label{ineq:ZCBF2}
	& \sup_{u \in U}  [L_g L_f^{r-1}h(x)u +L_f^rh(x)+{\bf a}'\eta(x) ] \geq 0
\end{align}
where $\eta(x)=[L_f^{r-1}h, L_f^{r-2}h,...,h]^{\top}\in\R^r$, and ${\bf a}=[a_1,...,a_r]^{\top}\in\R^r$ is chosen such that the roots of $p_0^r(\la)=\la^r+a_1\la^{r-1}+...+a_{r-1}\la+a_r$ are all negative reals $-\la_1,...,-\la_r<0$.
Define functions $s_k(x(t))$ for $k=0,1,...,r$ as follows:
\begin{align}\label{liftedh}
s_0(x(t))&=h(x(t)),\;s_{k}(x(t))=(\frac{\diff}{\diff t}+\la_k)s_{k-1}.
\end{align}
It was shown in \cite{nguyen2016exponential} that any controller $u(x) \in \{ u \in U : L_g L_f^{r-1}h_0(x)u +L_f^rh_0(x)+{\bf a}'\eta(x) ] \geq 0\}$ that is Lipschitz  will guarantee the forward invariance of $\mathcal{C}$. The time-varying CBF with a general relative degree and its safety guarantee for a time-varying system were discussed in \cite{xu2018constrained}.

\subsection{Estimation Error Quantified Observer }
An observer for system \eqref{eqnsys}-\eqref{output} is given as:
\begin{equation}
    \dot{\hat{x}}=v(\hat{x},y,u)\label{observer}
\end{equation}
where $\hat{x}$ is the estimated state, $u$ is the input, and $y$ is the output. Define the state estimation error as
\begin{align}
e=\hat x-x.\label{eqe}
\end{align}
Then the \emph{error dynamics} is given as
\begin{equation}\label{eqerror}
    \dot{e}=\bar{v}(e,\hat{x},y,u)
\end{equation}
where $\bar{v}(e,\hat{x},y,u)=v(\hat{x},y,u)-f(\hat x-e)-g(\hat x-e)u$.

In this paper, we consider the \emph{estimation error quantified (EEQ) observer} that is a generalization of the traditional asymptotic observer that requires the state estimation error to converge to zero. The definition of the EEQ observer is introduced below.
\begin{definition}
An observer is called an EEQ observer for system \eqref{eqnsys}-\eqref{output} if it provides a state estimation $\hat{x}(t)$ such that
\begin{equation}\label{observerdistance}
    \|\hat x(t) -x(t)\|\leq M(t,x_0,\hat x_0),\quad \forall t\geq 0,
\end{equation}
where  $M(t,x_0,\hat x_0): \R_{\geq 0}\times\mathbb{R}^n\times\mathbb{R}^n\rightarrow \R_{\geq 0}$ is a known time-varying function whose values are  non-negative.
\end{definition}

To simply the notation, we will use $M(t)$ for $M(t,x_0,\hat x_0)$ in the following. The EEQ observer subsumes many types of common observers.
\subsubsection{Interval Observers}
The interval observer is an observer that provides an estimation interval for the true states by using the input-output measurement  \cite{efimov2013interval,moisan2007near}.
Specifically, an interval observer has two dynamic systems that provide the upper bound $\bar{x}(t)$ and lower bound $\underline{x}(t)$ of the true state $x(t)$, respectively.
If its state estimation  is selected as $\hat{x}(t)=\frac{\bar{x}(t)+\underline{x}(t)}{2}$, then an interval observer is an EEQ observer with $M(t)$ shown in \eqref{observerdistance}  chosen as
\begin{equation}
    M(t)=\frac{1}{2}\|\bar{x}(t)-\underline{x}(t)\|.
\end{equation}

\subsubsection{Exponentially Stable Observers}
According to \cite[p.~150]{khalil2002nonlinear}, an (global) exponentially stable observer requires that the equilibrium point $e=0$ of the error system shown in \eqref{eqerror} is exponentially stable, i.e., there exist positive constants $k$, and $\lambda$ such that $\|e(t)\|\leq k\|e(0)\|\text{exp}(-\lambda t)$.
An exponentially stable observer is an EEQ observer with $M(t)$ shown in \eqref{observerdistance}  chosen as
\begin{equation}\label{expobserver}
    M(t)=D\text{exp}(-\lambda t)
\end{equation}
where $D$ is a positive constant satisfying $D\geq k\|e(0)\|$.
Note that any exponentially stable observer  can be employed in the proposed control scheme.

\subsubsection{Neural-Network-Based Observer}

Deep neural networks can be employed to design observers because of its universal approximation property \cite{marchi2021safety,tabuada2020universal,elkenawy2020diagonal}.
For example, in \cite{marchi2021safety}, a neural network is trained to approximate the function $\psi$, which recovers $x$ from $y$ by $x=\psi(y)$, such that the state estimation is given by a trained function as $\hat{x}=r(y,\alpha)$, where $\alpha$ is the training parameter.
There are many techniques that can be used to train the neural networks, such as stochastic gradient decent and Adam algorithm.
Nevertheless, since the neural network is trained on the training set $E_S$, its approximation accuracy on the complete dataset $E$ is not always guaranteed even if the approximation error is small enough on $E_S$. Marchi \etal \cite{marchi2021safety} pointed out that if $E_S$ is a $\delta$-cover of $E$, any continuous function $\psi$ can be approximated by a function $r=\phi+A$, where $\phi$ is a monotone function and $A$ is a linear map, with the generalization error
\begin{equation}\label{eqNNobser}
    \|\psi-r\|_{L^\infty(E)}\leq 3\|\psi-r\|_{L^\infty(E_S)}+2\omega_\psi(\delta)+2\|A\|_\infty\delta
\end{equation}
where $\omega_\psi$ is a modulus of continuity of $\psi$ on $E$ and $\|A\|_\infty$ denotes the operator $\infty$-norm of the map $A$.
Thus, the neural-network-based observer $r$ designed in \cite{marchi2021safety} is an EEQ observer with  $M(t)$ shown in \eqref{observerdistance}  chosen as
\begin{equation}
    M(t) = \beta
\end{equation}
where $\beta$ is the constant on the right hand side of \eqref{eqNNobser}.

\subsection{Problem Statement}\label{problemformulate}

This paper will consider the CBF-based safety control design problem with an EEQ observer in the loop. Specifically, the problem that will be studied is given as follows.
\begin{problem}\label{prob1}
\emph{Given system \eqref{eqnsys}-\eqref{output} and its EEQ observer \eqref{observerdistance}, design a feedback controller $u(\hat x): \R^n\rightarrow \R^m$ such that the trajectory of the closed-loop system will stay inside  the safe set $\mathcal{C}$ defined in \eqref{setc}, i.e., $h(x(t))\geq 0$ for all $t\geq 0$.}
\end{problem}

\section{Main Results}\label{sec:control}
In this section, we reconstruct system \eqref{eqnsys} whose state $x$ cannot be known exactly into a model of $\hat x$, which is the state of the EEQ observer, by using FAT introduced in Section \ref{sec:FAT}. Based on that, we develop an adaptive CBF method to design the safe controller for CBFs with a relative degree 1 and a higher relative degree individually.

Recalling that  the state estimation error $e$ is defined as shown in \eqref{eqe}, system \eqref{eqnsys} can be rewritten as
\begin{equation}
    \dot{\hat{x}}=
    f(\hat{x})+g(\hat{x})u+\Lambda(x,\dot{e},\hat{x},u)\label{trsystem}
\end{equation}
where
$$
\Lambda(x,\dot{e},\hat{x},u)=\dot{e}+\delta f+\delta g u\in\R^n
$$
with $\delta f=f(x)-f(\hat{x})\in\R^n$ and $\delta g=g(x)-g(\hat{x})\in\R^{n\times m}$. Since $x$, $\dot{e}$, $\hat{x}$ are solutions of the closed-loop system composed of systems \eqref{eqnsys}-\eqref{output} and its EEQ observer \eqref{observerdistance}, they are variables with respect to time $t$. Thus, $\Lambda(x,\dot{e},\hat{x},u)$ can also be seen as a function of $t$, which can be  approximated by using trigonometric functions as the basis function as follows \cite{wang2021function,bai2020function}:
\begin{equation}
    \Lambda (x,\dot{e},\hat{x},u)
    =\sum_{i=1}^N \theta_i \varphi_i(t)+\epsilon_\Lambda\label{fattheta}
\end{equation}
where $\varphi_i\in\R$ represents the set of trigonometric scalar functions defined in \eqref{dii}, $N$ is a positive integer, $\epsilon_\Lambda\in\R^n$ denotes the  truncation error, and $\theta_i\in\R^n$ is a vector of parameters that are unknown constants.
Substituting \eqref{fattheta} into \eqref{trsystem} yields
\begin{equation}
    \dot{\hat{x}}=
    f(\hat{x})+g(\hat{x})u+\sum_{i=1}^N\theta_i \varphi_i(t)+\epsilon_\Lambda.\label{trsystem2}
\end{equation}
It can be seen that the reconstructed system \eqref{trsystem2} is a model of $\hat{x}$ containing unknown parameters. We will solve Problem \ref{prob1} by considering \eqref{trsystem2}  and using an adaptive control design method. The control input $u$ will be designed to render the set $\mathcal{C}$ safe with regard to system \eqref{trsystem2} in the presence of unknown parameters $\theta_i$.
Two assumptions regarding $\epsilon_\Lambda$ and $\theta_i$ are proposed as follows.
\begin{assumption}\cite{izadbakhsh2019fat,huang2004adaptive,bai2021motion}\label{assumpeps}
There exists a positive constant $E>0$ such that $\|\epsilon_\Lambda\|\leq E$. 
\end{assumption}
\begin{assumption}\cite{ebeigbe2019robust}\label{assumptheta}
   There exist constants $\bar{\theta}_i$ for $i=1,...,N$, such that the unknown parameter $\theta_i$ in \eqref{trsystem2} is bounded by $\bar{\theta}_i$,  i.e.,
    $$
    \|\theta_i\|\leq\bar{\theta}_i.
    $$
\end{assumption}

\begin{remark}
Given an arbitrary constant $E>0$, one can choose $N$ large enough to make the truncation error  $\|\epsilon_{\Lambda}\|$ smaller than $E$. Thus, Assumption \ref{assumpeps} can be always satisfied by choosing a large enough $N$. Although a better approximation accuracy can be achieved with a smaller $E$, the corresponding  computational burden may grow up rapidly with the increase of $N$, which is not desirable in real applications. Moreover, if $N$ is too large, Gibbs phenomenon, which degrades the approximation accuracy and induces oscillations, may appear. Our past experience indicates that in most cases, $N\leq 5$ is sufficient to guarantee a good approximation accuracy. \hfill$\Box$

\end{remark}

\subsection{Safe Control Design for  CBF with Relative Degree 1}

In this subsection,  a feedback controller will be designed for system \eqref{trsystem2} to solve Problem \ref{prob1} where the CBF $h$ is assumed to have a relative degree 1.

Note that the state variable of system \eqref{trsystem2} is $\hat{x}$ instead of $x$, while $h(x)$ is a function of $x$.
Assume $h(x)$ is a global Lipschitz function. Thus,
there exists a constant $L>0$ as the Lipschitz constant of $h$, such that for all $x, \hat x$,
\begin{equation}
    |h(x)-h(\hat{x})|\leq L \|x-\hat{x}\|\label{Liph}
\end{equation}
where $x,\hat x$ are the true and estimated states of system \eqref{eqnsys}, respectively, which implies that
\begin{equation}
    h(x)\geq h(\hat{x})-L\|x-\hat{x}\|\geq
    h(\hat{x})-LM(t)\label{newh}
\end{equation}
where the last inequality is from \eqref{observerdistance}. Define a time-varying function $h:\mathbb{R}^n\times \mathbb{R}\to\mathbb{R}$ as
\begin{align}\label{eqh}
h_0(x,t)=h(x)-LM(t).
\end{align}
From \eqref{newh} and the fact that $M(t)\geq 0$ for $t\geq 0$, it is clear that $h_0(\hat{x},t)\geq 0$ implies $h(x)\geq 0$ for any $t\geq 0$. Therefore, if a controller can be designed such that $h_0(\hat x,t)\geq 0$  for $t\geq 0$, then the forward invariance of $\mathcal{C}$ will be guaranteed.

\begin{remark}\label{remarklipschitz}
Assume that the set of initial conditions $x_0,\hat x_0$ renders  $M(t)\leq M_0$ for $t\geq 0$, where $M_0>0$ is a fixed positive constant. Define $\mathcal{C}^{M_0}\triangleq\mathcal{C}\oplus B_{M_0}(0)$ where $\oplus$ is the Minkowski sum and $B_{M_0}(0)$ is the 2-norm ball at the origin with a radius of $M_0$.
Then the global Lipschitz requirement on $h$ shown in \eqref{Liph} can be relaxed to the condition that $h$ has a Lipschitz constant $L$ on $\mathcal{C}^{M_0}$.
\end{remark}

Define a new function $h_\epsilon:\R^n\times \R\rightarrow \R$ as
\begin{equation}\label{eqheps}
    h_\epsilon(\hat{x},t) = h_0(\hat{x},t)-\epsilon
\end{equation}
where $h_0$ is given in \eqref{eqh}, and $\epsilon>0$ is a positive constant that will be determined later.  The following theorem is the main result of this subsection.

\begin{theorem}\label{theorem1}
Consider system \eqref{trsystem2} with unknown parameters $\theta_i$ and a set $\mathcal{C}$ defined in \eqref{setc} for a continuously differentiable function $h$. Suppose that $h$ has a relative degree 1 and satisfies \eqref{Liph} for some $L>0$. Suppose that Assumption \ref{assumpeps} and \ref{assumptheta} hold. Suppose that the parameter estimation $\hat\theta_i$ is  governed by the following adaptive law:
\begin{equation}\label{adaptivelaw}
    \dot{\hat{\theta}}_i = -\frac{\bar{\theta}_i^2}{2\epsilon}\frac{\partial  h_\epsilon}{\partial \hat{x}}\varphi_i-\mu\hat{\theta}_i,
\end{equation}
where $\mu>0$ a  positive constant, and $h_\epsilon$ is given in \eqref{eqheps}. If $\epsilon$ is chosen such that
\begin{equation}\label{initialcondition}
   0<\epsilon \leq  \frac{h_0(\hat{x}(0),0)}{ \bigg(N+\sum_{i=1}^N\bigg(\frac{2\|\hat{\theta}_i(0)\|}{\bar{\theta}_i}+\frac{\|\hat{\theta}_i(0)\|^2}{\bar{\theta}_i^2} \bigg) \bigg)},
\end{equation}
then any Lipschitz continuous controller  $u(x) \in K_{BF}(\hat{x},\hat{\theta},t)$  where
\begin{IEEEeqnarray}{rCl}
    \hspace{-6mm}K_{BF}(\hat{x},\hat{\theta},t)\triangleq\bigg\{u\in\mathbb{R}^m \, | \, && L_g h_\epsilon u+L_{\tilde{f}} h_\epsilon-\mu N\epsilon+\frac{\partial h_\epsilon}{\partial t}\nonumber\\
    &&-\bigg\|{\frac{\partial h_\epsilon}{\partial \hat{x}}}\bigg\|E+\mu h_\epsilon\geq 0\bigg\},\label{controlinput}
\end{IEEEeqnarray}
with $\tilde{f}=f+\sum_{i=1}^N \hat{\theta}_i\varphi_i$ will guarantee the safety of $\mathcal{C}$ in regard to system \eqref{trsystem2}.
\end{theorem}

\begin{proof}
Define a composite CBF candidate as
\begin{equation}
    \bar{h}(\hat{x},\hat{\theta},t)=h_0(\hat{x},t)-\sum_{i=1}^N\frac{\epsilon}{\bar{\theta}_i^2}\tilde{\theta}_i^\top\tilde{\theta}_i,\label{barh}
\end{equation}
where $\tilde{\theta}_i=\theta_i-\hat{\theta}_i$ represents the parameter estimation error and $\hat{\theta}=[\hat{\theta}_1,\cdots,\hat{\theta}_N]^T$.
The time derivative of $\bar{h}$ is
\begin{IEEEeqnarray}{rCl}
&&\dot{\bar{h}}(\hat{x},\hat{\theta},t)\nonumber \\
&&={\frac{\partial h_0}{\partial \hat{x}}}^\top\!(f(\hat x)
\!+\!g(\hat x)u\!+\!\sum_{i=1}^N\theta_i\varphi_i\!+\!\epsilon_\Lambda)
\!+\!\frac{\partial h_0}{\partial t}\!+\!\sum_{i=1}^N\frac{2\epsilon}{\bar{\theta}_i^2}
\tilde{\theta}_i^\top\dot{\hat{\theta}}_i\nonumber\\
&&={\frac{\partial h_\epsilon}{\partial \hat{x}}}^\top
(f(\hat x)+g(\hat x)u+\sum_{i=1}^N\hat{\theta}_i\varphi_i
+\epsilon_\Lambda)+\frac{\partial h_\epsilon}{\partial t}\nonumber\\
&&\quad +\sum_{i=1}^N\tilde{\theta}_i^\top\bigg(\frac{\partial h_\epsilon}{\partial \hat{x}}\varphi_i+\frac{2\epsilon}{\bar{\theta}_i^2}\dot{\hat{\theta}}_i \bigg).\label{dotbarh}
\end{IEEEeqnarray}
Substituting \eqref{adaptivelaw} into \eqref{dotbarh} gives
\begin{equation*}
    \dot{\bar{h}}(\hat{x},\hat{\theta},t)=L_g h_\epsilon u+L_{\tilde{f}} h_\epsilon+{{\frac{\partial h_\epsilon}{\partial \hat{x}}}}^\top\epsilon_\Lambda+\frac{\partial h_\epsilon}{\partial t}-\sum_{i=1}^N \frac{2\mu\epsilon}{\bar{\theta}_i^2}\tilde{\theta}_i^\top\hat{\theta}_i.
\end{equation*}
It can be seen that $u\in K_{BF}(\hat{x},\hat{\theta},t)$ indicates
\begin{IEEEeqnarray}{rCl}
    \dot{\bar{h}}(\hat{x},\hat{\theta},t)
    &\geq&\! {{\frac{\partial h_\epsilon}{\partial \hat{x}}}}^\top\!\epsilon_\Lambda\!+\!\bigg\|{\frac{\partial h_\epsilon}{\partial \hat{x}}}\bigg\|\!E-\mu  h_\epsilon\!-\!\sum_{i=1}^N\! \frac{2\mu\epsilon}{\bar{\theta}_i^2}\tilde{\theta}_i^\top\hat{\theta}_i\!+\!\mu N\epsilon\nonumber\\
    &\geq& -\mu h_0(\hat{x},t)+\mu (N+1) \epsilon-\sum_{i=1}^N \frac{2\mu\epsilon}{\bar{\theta}_i^2}\tilde{\theta}_i^\top\hat{\theta}_i\label{tilhdot}
\end{IEEEeqnarray}
where the first inequality is from \eqref{controlinput}, and the second inequality is derived from the assumption that $\|\epsilon_\Lambda\|\leq E$, which is stated in Assumption \ref{assumpeps}.
Since
\begin{align*}
\tilde{\theta}_i^\top\hat{\theta}_i&=\tilde{\theta}_i^\top(\theta_i-\tilde{\theta}_i)    \leq\frac{\theta_i^\top\theta_i}{2}-\frac{\tilde{\theta}_i^\top\tilde{\theta}_i}{2},
\end{align*}
from \eqref{tilhdot} one gets
\begin{equation}
    \dot{\bar{h}}(\hat{x},\hat{\theta},t)
    \geq -\mu  \bar{h}(\hat{x},\hat{\theta},t)+\mu (N+1) \epsilon-\sum_{i=1}^N \frac{\mu\epsilon}{\bar{\theta}_i^2}\theta_i^\top\theta_i.
\end{equation}
As $\|\theta_i\|\leq\bar{\theta}_i$, one obtains
\begin{equation}
    \dot{\bar{h}}(\hat{x},\hat{\theta},t)
    \geq -\mu  \bar{h}(\hat{x},\hat{\theta},t).\label{dotbarhinequality}
\end{equation}
By the comparison lemma \cite[Page 103]{khalil2002nonlinear}, we get
\begin{equation}
    \bar{h}(t)\geq \bar{h}(\hat{x}(0),\hat{\theta}(0),0)e^{-\mu t}.\label{eqbarh}
\end{equation}
Since
\begin{align}
   \bar{h}(\hat{x},\hat{\theta},t)&= h_0(\hat{x},t)\!-\!\sum_{i=1}^N\frac{\epsilon}{\bar{\theta}_i^2}(\hat{\theta}_i^\top\hat{\theta}_i-2\theta_i^\top\hat{\theta}_i+\theta_i^\top\theta_i)\nonumber\\
   &\geq h_0(\hat{x},t)\!-\!\sum_{i=1}^N\frac{\epsilon}{\bar{\theta}_i^2}(\|\hat{\theta}_i\|^2+2\|\hat{\theta}_i\|\bar{\theta}_i+\bar{\theta}_i^2),\label{initialsystem}
\end{align}
when $t=0$, substituting \eqref{initialcondition} into \eqref{initialsystem} yields
\begin{IEEEeqnarray}{rCl}
    \bar{h}(\hat{x}(0),\hat{\theta}(0),0)&\geq&\epsilon \bigg(N+\sum_{i=1}^N\bigg(\frac{2\|\hat{\theta}_i(0)\|}{\bar{\theta}_i}+\frac{\|\hat{\theta}_i(0)\|^2}{\bar{\theta}_i^2} \bigg) \bigg)\nonumber\\
    &&-\sum_{i=1}^N\frac{\epsilon}{\bar{\theta}_i^2}(\|\hat{\theta}_i(0)\|^2+2\|\hat{\theta}_i(0)\|\bar{\theta}_i+\bar{\theta}_i^2)\nonumber\\
    &=&0.\label{eqbarh0}
\end{IEEEeqnarray}
Therefore, from \eqref{eqbarh} and \eqref{eqbarh0} one gets $\bar{h}(t)\geq 0$ for $t\geq 0$.
From \eqref{barh}, it can be seen that $\bar{h}(t)\geq 0$ indicates $h_0(\hat x,t)\geq 0$. Moreover, $h_0(\hat x,t)\geq 0$ implies $h(x)\geq 0$. Hence, the set $\mathcal{C}$ is safe with regard to system \eqref{trsystem2}.
\end{proof}

By Theorem \ref{theorem1}, the safe controller for solving Problem \ref{prob1} is obtained by  the following CBF-QP:
\begin{align}
\min_{u} \quad & \|u-u_d\|^2\label{cbfQP}\tag{CBF-QP1}\\
\textrm{s.t.} \quad & L_g h_\epsilon u+L_{\tilde{f}} h_\epsilon+\frac{\partial h_\epsilon}{\partial t}-\bigg\|{\frac{\partial h_\epsilon}{\partial \hat{x}}}\bigg\|E+\mu h_\epsilon-\mu N \epsilon\geq 0\nonumber
\end{align}
where $u_d$ represents a nominal control that may be unsafe control law,
$\tilde{f}=f+\sum_{i=1}^N \hat{\theta}_i\varphi_i$, and $\hat{\theta}_i$ is updated by the adaptive law shown in \eqref{adaptivelaw}.

\subsection{Safe Control Design for  CBF with High Relative Degree}
In this subsection,  we will consider solving Problem \ref{prob1} for a CBF $h$ that is assumed to have a relative degree $r\geq 2$.
Recall the definition of functions $s_k(x(t))$ shown in \eqref{liftedh}. Assume that $s_k(x(t))$ is globally Lipschitz and has Lipschitz constant $L_k>0$ for $k=0,1,\cdots,r-1$. Note that this requirement can be relaxed as metioned in Remark \ref{remarklipschitz}.
Define a family of sets:
\begin{equation}
    \mathcal{C}_k=\{x\in\mathbb{R}^n:  s_k(x)\geq 0\}, \ k=0,1,\cdots,r-1,
\end{equation}
and a family of functions:
\begin{equation}
s_k^M(x,t)=s_k(x)-L_k M(t), \ k=0,1,\cdots,r-1. \label{skm}\\
\end{equation}
\begin{theorem}\label{theorem2}
Consider the system \eqref{trsystem2} with unknown parameters $\theta_i$ and a set  $\mathcal{C}$ defined in \eqref{setc} for a $C^r$ function $h$ that has a relative degree $r$. Suppose that $s_k(x)$ has Lipschitz constant $L_k$  for $k=0,1,...,r$.  Suppose that Assumption \ref{assumpeps} and \ref{assumptheta} hold and the parameter estimation $\hat\theta_i$ is  governed by the following adaptive law:
\begin{equation}\label{adaptivelawhighre}
    \dot{\hat{\theta}}_i = -\frac{\bar{\theta}_i^2}{2\epsilon}\frac{\partial  s_\epsilon}{\partial \hat{x}}\varphi_i-\mu\hat{\theta}_i,
\end{equation}
where $\mu>0$ a  positive constant and $s_\epsilon(\hat x,t)=s_{r-1}^M(\hat x,t)-\epsilon$ with $s_k^M$ defined in \eqref{skm}. If there exists $\epsilon$ such that
\begin{equation}
    0<\epsilon\leq\frac{\min \{ s_0^M(\hat{x}(0),0),s_1^M(\hat{x}(0),0),\cdots,s_{r-1}^M(\hat{x}(0),0)\}}{\bigg(N+\sum_{i=1}^N\bigg(\frac{2\|\hat{\theta}_i(0)\|}{\bar{\theta}_i}+\frac{\|\hat{\theta}_i(0)\|^2}{\bar{\theta}_i^2} \bigg) \bigg)},\label{initialconditionhighre}
\end{equation}
then any Lipschitz continuous controller $u(x) \in K_{BF}(\hat{x},\hat{\theta},t)$  where
\begin{IEEEeqnarray}{rCl}\label{end}
    \hspace{-6mm}K_{BF}(\hat{x},\hat{\theta},t)\triangleq\bigg\{u\in\mathbb{R}^m \, | \, && L_g s_\epsilon u+L_{\tilde{f}} s_\epsilon-\mu N\epsilon+\frac{\partial s_\epsilon}{\partial t}\nonumber\\
    &&-\bigg\|{\frac{\partial s_\epsilon}{\partial \hat{x}}}\bigg\|E+\mu s_\epsilon\geq 0\bigg\},
\end{IEEEeqnarray}
with $\tilde{f}=f+\sum_{i=1}^N \hat{\theta}_i\varphi_i$ will guarantee the forward invariance of $\mathcal{C}$.
\end{theorem}
\begin{proof}
Similar to \eqref{newh} and \eqref{eqh}, it can be seen that $s_k^M(\hat{x},t)\geq 0$ indicates $s_k(x)\geq 0$.
From \eqref{initialconditionhighre}, one gets $s_k^M(\hat{x}(0),0)\geq 0$, such that $s_k(x(0))\geq 0$.
Following the same procedure as shown in the proof of Theorem \ref{theorem1}, one gets $s_{r-1}(x)\geq 0$. As $s_k(x(0))\geq 0$ holds for each $k=0,1,\cdots,r-1$, according to \cite[Theorem 1]{nguyen2016exponential}, one has that $\mathcal{C}_k$ is forward invariant. Since $\mathcal{C}=\mathcal{C}_0$, the conclusion follows immediately.
\end{proof}

By Theorem \ref{theorem2}, the safe controller for solving Problem \ref{prob1} where the CBF has a relative degree $r\geq 2$ is obtained by  the following CBF-QP:
\begin{align}
\min_{u} \quad & \|u-u_d\|^2\label{cbfQP2}\tag{CBF-QP2}\\
\textrm{s.t.} \quad & L_g s_\epsilon u+L_{\tilde{f}} s_\epsilon+\frac{\partial s_\epsilon}{\partial t}
    -\bigg\|{\frac{\partial s_\epsilon}{\partial \hat{x}}}\bigg\|E+\mu s_\epsilon-\mu N \epsilon\geq 0\nonumber
\end{align}
where $u_d$ represents a nominal control, $\tilde{f}=f+\sum_{i=1}^N \hat{\theta}_i\varphi_i$, and $\hat{\theta}_i$ is updated by the adaptive law shown in \eqref{adaptivelawhighre}.

\section{Simulation}
\label{sec:simulation}
In this section, we use two numerical examples to illustrate the effectiveness of the proposed control  scheme.

\begin{example}\label{example1}
Consider a linear system
\begin{IEEEeqnarray*}{rCl}
    \dot{x}&=& \underbrace{\begin{bmatrix}
    -1&2&-2\\
    0&-1&1\\
    1&0&-1
    \end{bmatrix}}_{A}x+\underbrace{\begin{bmatrix}
    0\\1\\1
    \end{bmatrix}}_{B}u,\\
    y&=&\underbrace{\begin{bmatrix}
    1&1&0
    \end{bmatrix}}_{C}x.
\end{IEEEeqnarray*}
A Luenberger observer is designed as $\dot{\hat{x}}=A\hat{x}+Bu+L(y-C\hat{x})$
where $L=\begin{bmatrix}
-2.23029 & 0.190287 & 0.232326
\end{bmatrix}^\top$ such that $A-LC$ is Hurwitz. Consider the safe set $\mathcal{C}\triangleq\{ x =[x_1,x_2,x_3]^\top \in \R^3 : x_2\geq 1\}$, where the  corresponding CBF is given as
\begin{equation}
    h(x)=x_2-1.\label{ex1cbf1}
\end{equation}
Note that $h(x)$ shown in \eqref{ex1cbf1} has a relative degree 1. The initial conditions are chosen as $x(0)=\begin{bmatrix}
2 &2.2 &2
\end{bmatrix}^\top$, $\hat x(0)=\begin{bmatrix}
3 &3.5 &3
\end{bmatrix}^\top$, the parameters are selected as $E=0.1$, $N=3$, $\bar{\theta}_i=0.5$, $\hat{\theta}_i(0)=\begin{bmatrix}
0 &0 &0
\end{bmatrix}^\top$, $\epsilon=0.1$, $\mu=3.5$, and $M(t)$ is an exponential function as shown in \eqref{expobserver} with $D=2$ and $\lambda=-0.05$. It can be verified that the condition \eqref{initialcondition} in Theorem \ref{theorem1} is satisfied.
The safe controller is obtained by solving  \eqref{cbfQP}. The evolution of CBF $h$ is shown in Fig. \ref{fig3} (a). Also shown is the evolution of $h$  by utilizing the estimated state $\hat{x}$ as the true state in the traditional CBF-QP in \cite{Xu2015ADHS}.

Now consider another safe set $\mathcal{C}\triangleq\{ x=[x_1,x_2,x_3]^\top \in \R^3 : x_1\geq 1\}$, where the  corresponding CBF is given as
\begin{equation}
    h(x)=x_1-1.\label{ex1cbf2}
\end{equation}
Note that $h(x)$ shown in \eqref{ex1cbf2} has a relative degree 2. The safe controller is obtained by solving  \eqref{cbfQP2}, where the parameters are selected as $E=0.1$, $N=3$, $\bar{\theta}_i=0.5$, $\hat{\theta}_i(0)=\begin{bmatrix}
0 &0 &0
\end{bmatrix}^\top$, $\epsilon=0.1$, $\mu=10$, and $\lambda_1=2$. The initial conditions are $x(0)=\begin{bmatrix}
2.4 &-3 &-3
\end{bmatrix}^\top$, $\hat x(0)=\begin{bmatrix}
3.4 &-2 &-2
\end{bmatrix}^\top$, and $M(t)$ is the same as above. It is easy to check the inequality \eqref{initialconditionhighre} holds true.
The simulation result is presented in Fig. \ref{fig3} (b).

From the simulation results, it can be seen that regardless of the relative degree of $h(x)$, the set $\mathcal{C}$ is forward invariant when the proposed approach is used, whereas the safety constraint  is violated if the estimated state, $\hat x$, is directly used as the true state in the traditional CBF-QP control scheme.
\end{example}

\begin{figure*}[!b]
\centering
 \begin{subfigure}[b]{0.9\textwidth}
 \centering
    \includegraphics[width=0.45\textwidth]{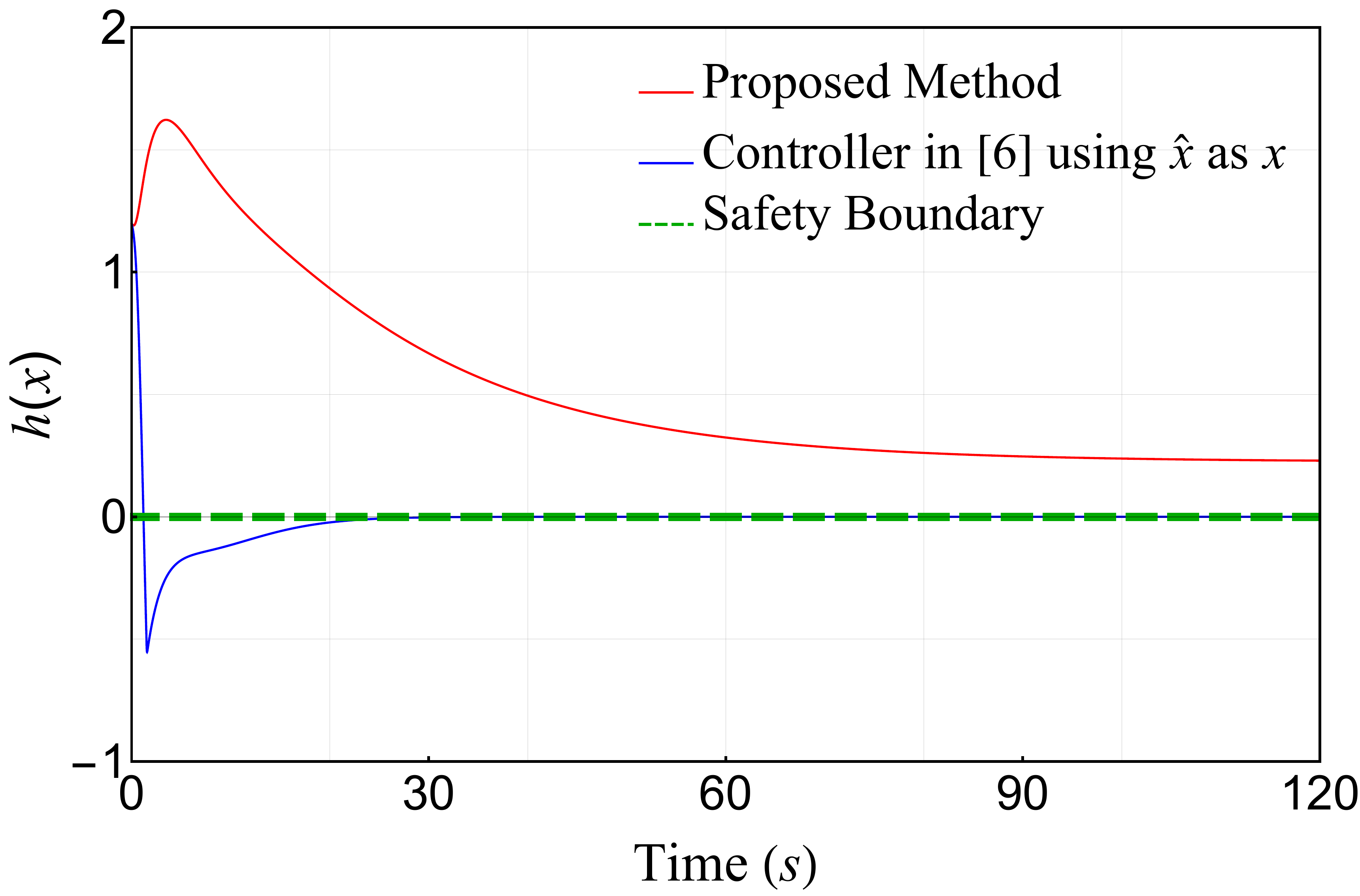}
    \includegraphics[width=0.45\textwidth]{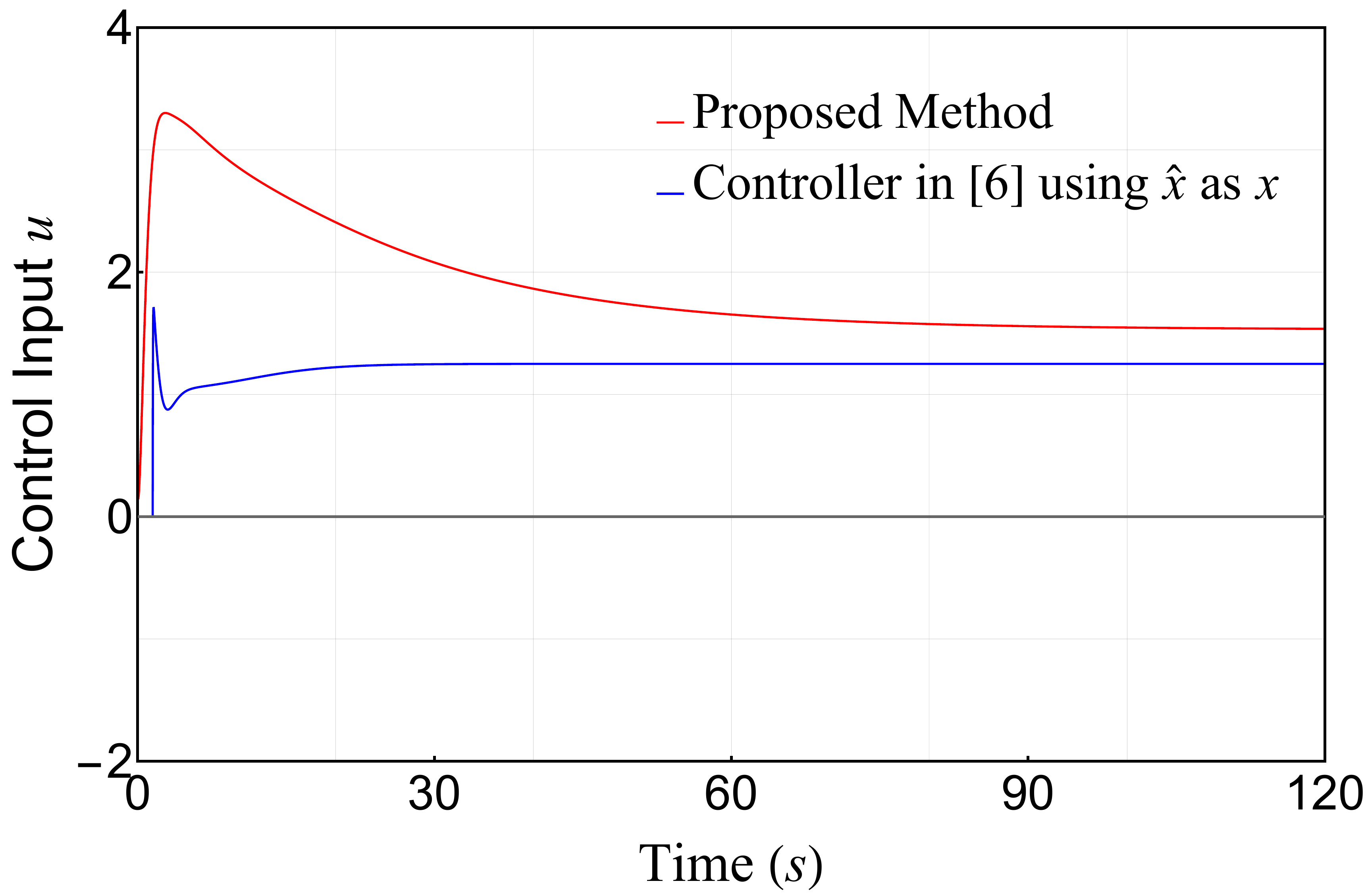}
    \caption{Simulation with CBF $h(x)=x_2-1$ that has a relative degree 1}
    \label{quad3d50}
  \end{subfigure}
  \begin{subfigure}[b]{0.9\textwidth}
  \centering
    \includegraphics[width=0.45\textwidth]{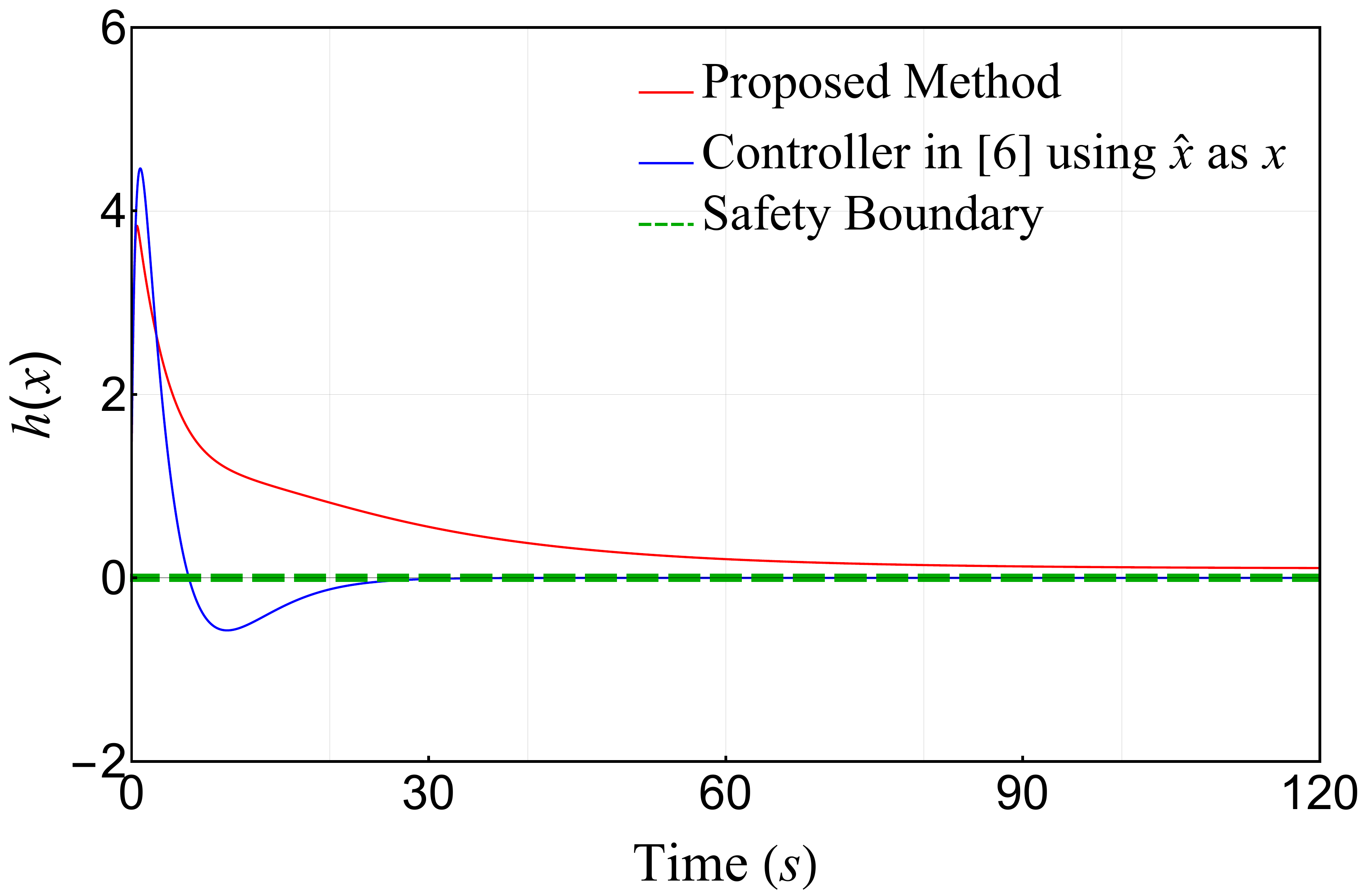}
    \includegraphics[width=0.45\textwidth]{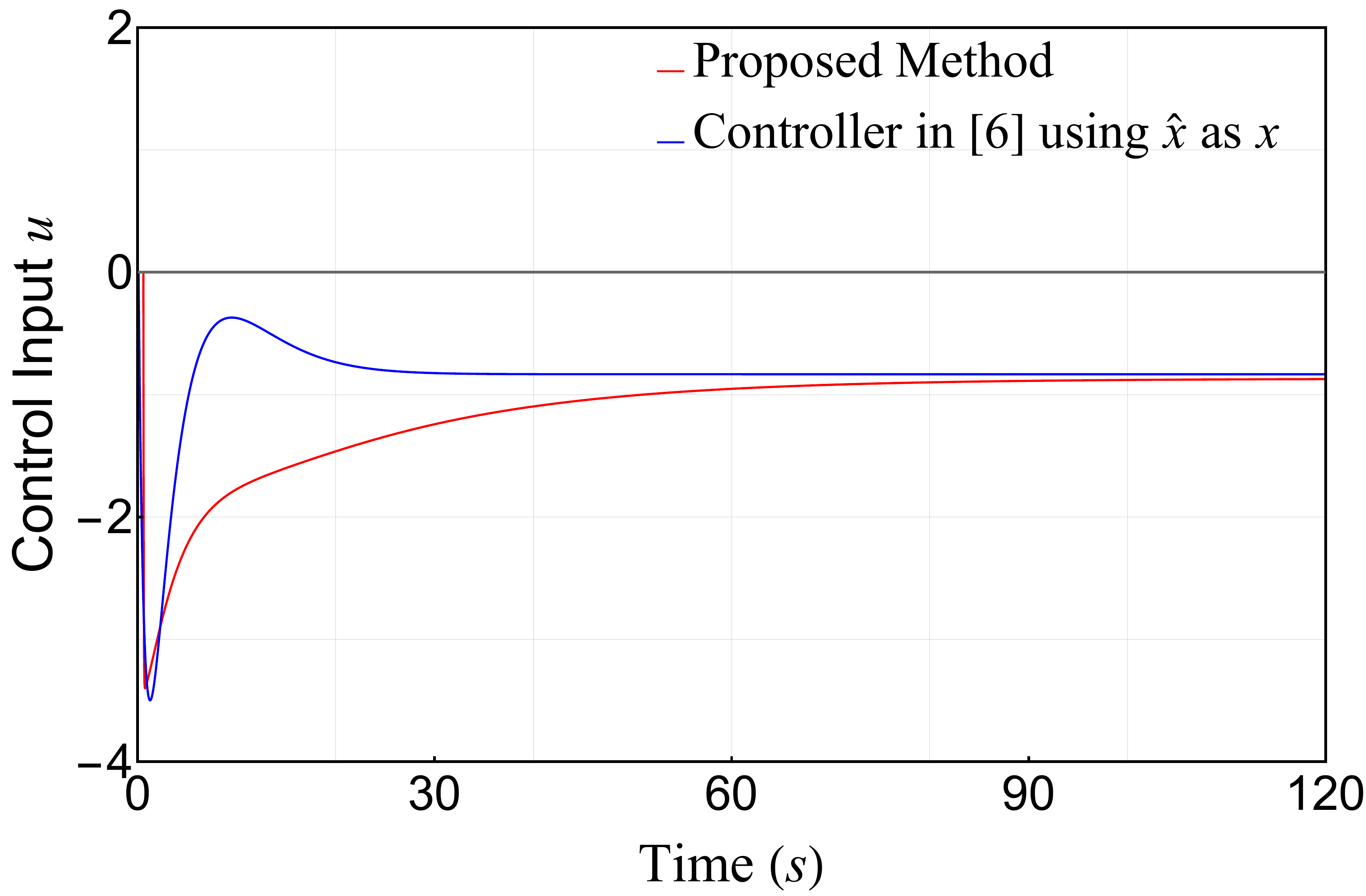}
    \caption{Simulation with CBF $h(x)=x_1-1$ that has a relative degree 2}
    \label{quad3d100}
  \end{subfigure}
  \caption{Simulation results of Example \ref{example1}. For either CBF, the values of $h$ is always positive when implementing the proposed CBF-QP  controller, meaning that safety is always ensured; in contrast,  safety is violated when utilizing the estimated state, $\hat x$, as the true state in the traditional CBF-QP controller.}
  \label{fig3}
\end{figure*}

\begin{example}\label{example2}
Consider the R\"{o}ssler system:
\begin{equation*}
    \begin{bmatrix}
    \dot x_1\\\dot x_2\\\dot x_3
    \end{bmatrix}=
    \begin{bmatrix}
    -x_2-x_3\\
    x_1+ax_2\\
    b+x_3(x_1-c)
    \end{bmatrix}+\bm{I}_{3\times 3} u,
\end{equation*}
where $a=b=0.2$, $c=5$ are chosen as same as those in \cite{mata2010exponential}, $u=\begin{bmatrix}
u_1&u_2&u_3
\end{bmatrix}$ denotes the control input and $\bm{I}_{3\times 3}$ is the identity matrix.
An exponentially stable observer is designed as in \cite{mata2010exponential}:
\begin{eqnarray*}
\begin{cases}
\dot{\hat{x}}_1\!=\!-\hat{x}_2\!-\!\hat{x}_3\!+\!q_1(x_1\!-\!\hat{x}_1)\!+\!q_2(x_1\!-\!\hat{x}_1)^m\!+\!u_1,\\
\dot{\hat{x}}_2\!=\!\hat{x}_1\!+\!a\hat{x}_2\!+\!s_1(x_1\!-\!\hat{x}_1)\!+\!s_2(x_1\!-\!\hat{x}_1)^m\!+\!u_2,\\
\dot{\hat{x}}_3\!=\!b\!+\!\hat{x}_3(\hat{x}_1\!-\!c)\!+\!r_1(x_1\!-\!\hat{x}_1)\!+r_2(x_1\!-\!\hat{x}_1)^m\!+\!u_3,\!\!\!\!\!\!\!\!\!
\end{cases}
\end{eqnarray*}
where
$q_1=3$, $s_1=-3$, $r_1=3$, $q_2=s_2=r_2=10$, and $m=3$. Consider the safe set $\mathcal{C}\triangleq\{ x=[x_1,x_2,x_3]^\top \in \R^3 : x_2\geq -1\}$, where the  corresponding CBF is given as
\begin{equation*}
    h(x)=x_2+1
\end{equation*}
which has a relative degree 1. The initial conditions are given as
$x(0)=\begin{bmatrix}
-0.5 &0.5 &3
\end{bmatrix}^\top$, $\hat x(0)=\begin{bmatrix}
0.2&2 &3
\end{bmatrix}^\top$, the parameters are selected as $E=0.1$, $N=3$, $\bar{\theta}_i=0.5$, $\hat{\theta}_i(0)=\begin{bmatrix}
0 &0 &0
\end{bmatrix}^\top$, $\epsilon=0.1$, $\mu=2.5$, and $M(t)$ is an exponential function as shown in \eqref{expobserver} with $D=2$ and $\lambda=-0.15$.
The simulation results are depicted in Fig. \ref{chaotic}, from which it can be seen that safety of the system is satisfied by the proposed CBF-QP controller in the presence of state estimation error while safety is violated by the traditional CBF-QP controller that uses the estimated state, $\hat x$, as the true state.

\end{example}

\begin{figure*}
  \centering
  \includegraphics[width=70mm]{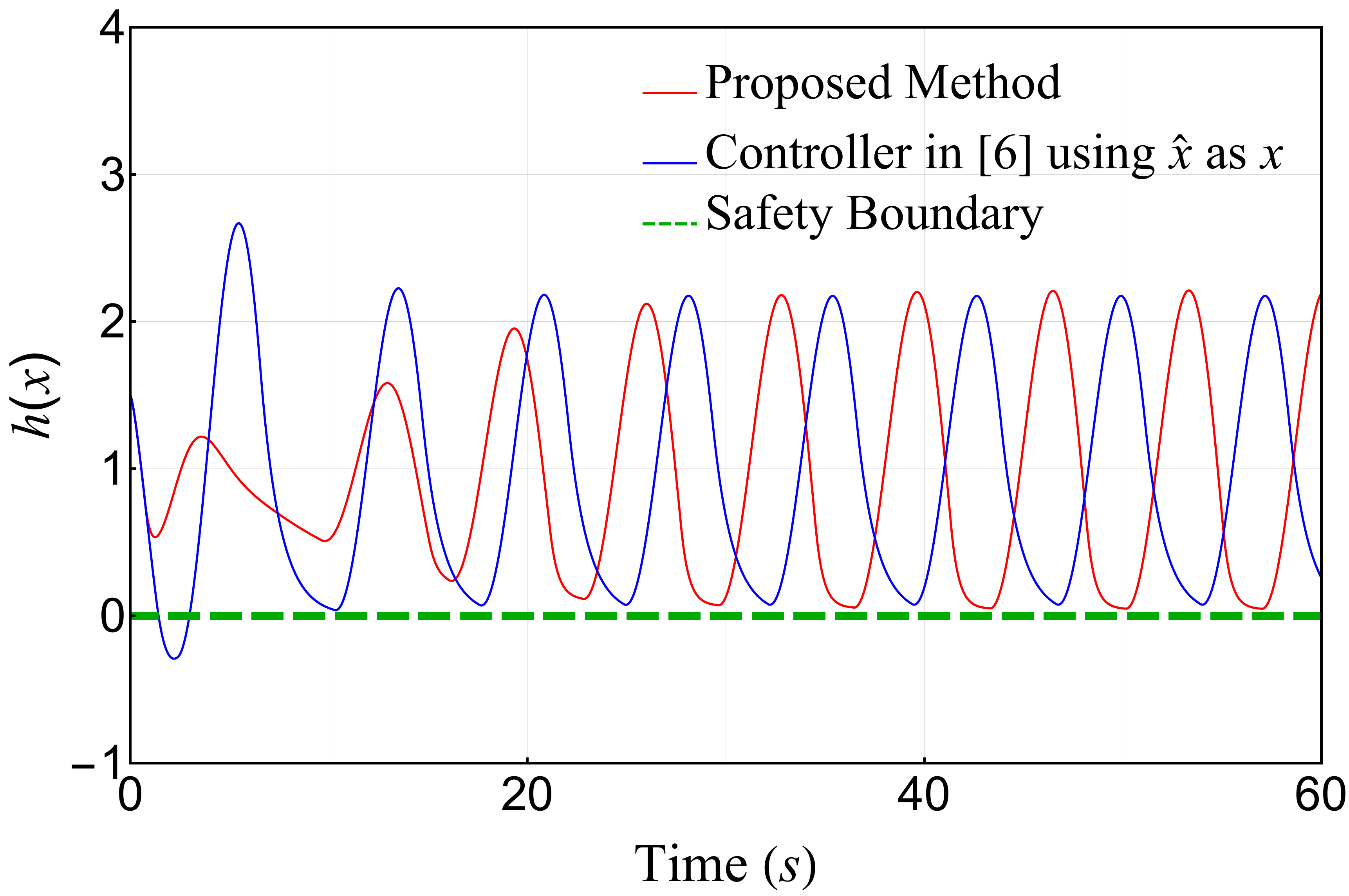}\hskip 5mm
  \includegraphics[width=70mm]{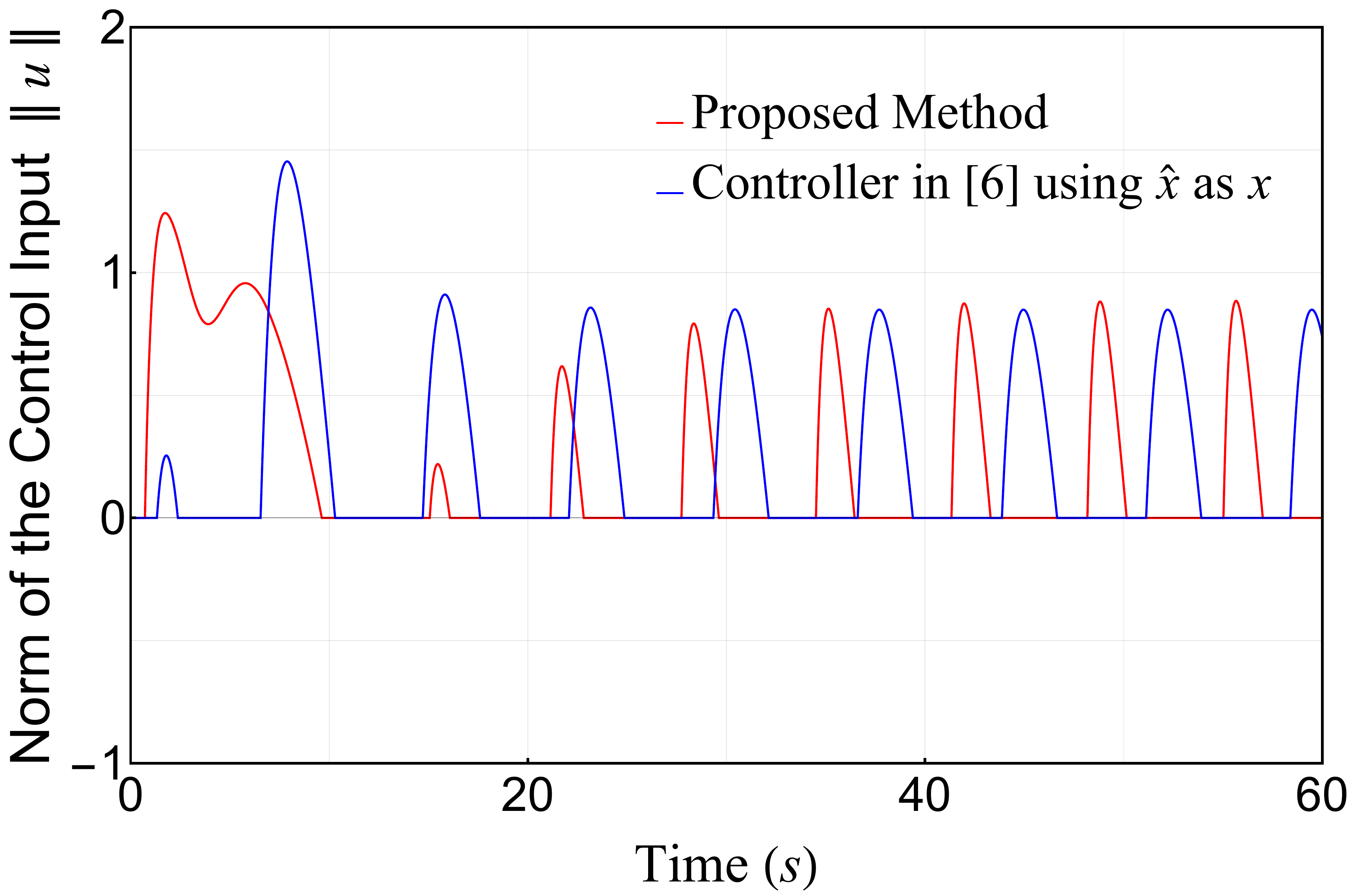}
  \caption{Simulation results of Example \ref{example2}. The CBF is  $h(x)=x_2+1$ that has a relative degree  1. The values of $h$ is always positive when implementing the proposed CBF-QP  controller, meaning that safety is always ensured, while safety is violated when the estimated state, $\hat x$, is used as the true state in the traditional CBF-QP controller. }\label{chaotic}
\end{figure*}

\section{Conclusions}
\label{sec:concl}
In this paper, a novel control framework that combines the EEQ observer and the CBF approach is proposed for safety-critical control systems with imperfect state measurements. The uncertainties introduced by state estimation error is approximated by FAT,
and the adaptive CBF technique is employed to design controllers via quadratic programs. The proposed control strategy is validated by numerical simulations. Future studies include developing EEQ observer design techniques using deep neural networks  and relaxing assumptions of this work.



\bibliographystyle{IEEEtran}


\begin{thebibliography}{10}
\bibitem{aswani2013provably}
A.~Aswani, H.~Gonzalez, S.~S. Sastry, and C.~Tomlin, ``Provably safe and robust
  learning-based model predictive control,'' \emph{Automatica}, vol.~49, no.~5,
  pp. 1216--1226, 2013.

\bibitem{mitchell2005time}
I.~M. Mitchell, A.~M. Bayen, and C.~J. Tomlin, ``A time-dependent
  hamilton-jacobi formulation of reachable sets for continuous dynamic games,''
  \emph{IEEE Transactions on Automatic Control}, vol.~50, no.~7, pp. 947--957,
  2005.

\bibitem{althoff2014online}
M.~Althoff and J.~M. Dolan, ``Online verification of automated road vehicles
  using reachability analysis,'' \emph{IEEE Transactions on Robotics}, vol.~30,
  no.~4, pp. 903--918, 2014.

\bibitem{dong2017safety}
H.~Dong, Q.~Hu, and M.~R. Akella, ``Safety control for spacecraft autonomous
  rendezvous and docking under motion constraints,'' \emph{Journal of Guidance,
  Control, and Dynamics}, vol.~40, no.~7, pp. 1680--1692, 2017.

\bibitem{ames2016control}
A.~D. Ames, X.~Xu, J.~W. Grizzle, and P.~Tabuada, ``Control barrier function
  based quadratic programs for safety critical systems,'' \emph{IEEE
  Transactions on Automatic Control}, vol.~62, no.~8, pp. 3861--3876, 2016.

\bibitem{Xu2015ADHS}
X.~Xu, P.~Tabuada, A.~Ames, and J.~Grizzle, ``Robustness of control barrier
  functions for safety critical control,'' in \emph{IFAC Conference on Analysis
  and Design of Hybrid Systems}, vol.~48, no.~27, 2015, pp. 54--61.

\bibitem{xu2017correctness}
X.~Xu, J.~W. Grizzle, P.~Tabuada, and A.~D. Ames, ``Correctness guarantees for
  the composition of lane keeping and adaptive cruise control,'' \emph{IEEE
  Transactions on Automation Science and Engineering}, vol.~15, no.~3, pp.
  1216--1229, 2018.

\bibitem{hsu2015control}
S.-C. Hsu, X.~Xu, and A.~D. Ames, ``Control barrier function based quadratic
  programs with application to bipedal robotic walking,'' in \emph{American
  Control Conference (ACC)}.\hskip 1em plus 0.5em minus 0.4em\relax IEEE, 2015,
  pp. 4542--4548.

\bibitem{nguyen20163d}
Q.~Nguyen, A.~Hereid, J.~W. Grizzle, A.~D. Ames, and K.~Sreenath, ``3{D}
  dynamic walking on stepping stones with control barrier functions,'' in
  \emph{IEEE Conference on Decision and Control (CDC)}, 2016, pp. 827--834.

\bibitem{wang2017safe}
L.~Wang, A.~D. Ames, and M.~Egerstedt, ``Safe certificate-based maneuvers for
  teams of quadrotors using differential flatness,'' in \emph{IEEE
  International Conference on Robotics and Automation (ICRA)}, 2017, pp.
  3293--3298.

\bibitem{dean2020guaranteeing}
S.~Dean, A.~J. Taylor, R.~K. Cosner, B.~Recht, and A.~D. Ames, ``Guaranteeing
  safety of learned perception modules via measurement-robust control barrier
  functions,'' \emph{arXiv:2010.16001}, 2020.

\bibitem{abu2021feedback}
M.~Abu-Khalaf, S.~Karaman, and D.~Rus, ``Feedback from pixels: Output
  regulation via learning-based scene view synthesis,'' in \emph{Learning for
  Dynamics and Control}.\hskip 1em plus 0.5em minus 0.4em\relax PMLR, 2021, pp.
  828--841.

\bibitem{poonawala2021training}
H.~A. Poonawala, N.~Lauffer, and U.~Topcu, ``Training classifiers for feedback
  control with safety in mind,'' \emph{Automatica}, vol. 128, p. 109509, 2021.

\bibitem{takano2018robust}
R.~Takano and M.~Yamakita, ``Robust constrained stabilization control using
  control lyapunov and control barrier function in the presence of measurement
  noises,'' in \emph{Conference on Control Technology and Applications
  (CCTA)}.\hskip 1em plus 0.5em minus 0.4em\relax IEEE, 2018, pp. 300--305.

\bibitem{efimov2013control}
D.~Efimov, T.~Raissi, and A.~Zolghadri, ``Control of nonlinear and lpv systems:
  interval observer-based framework,'' \emph{IEEE Transactions on Automatic
  Control}, vol.~58, no.~3, pp. 773--778, 2013.

\bibitem{marchi2021safety}
M.~Marchi, J.~Bunton, B.~Gharesifard, and P.~Tabuada, ``Safety and stability
  guarantees for control loops with deep learning perception,'' \emph{IEEE
  Control Systems Letters}, 2021.

\bibitem{elkenawy2020diagonal}
A.~Elkenawy, A.~M. El-Nagar, M.~El-Bardini, and N.~M. El-Rabaie, ``Diagonal
  recurrent neural network observer-based adaptive control for unknown
  nonlinear systems,'' \emph{Transactions of the Institute of Measurement and
  Control}, vol.~42, no.~15, pp. 2833--2856, 2020.

\bibitem{huang2001sliding}
A.-C. Huang and Y.-S. Kuo, ``Sliding control of non-linear systems containing
  time-varying uncertainties with unknown bounds,'' \emph{International Journal
  of Control}, vol.~74, no.~3, pp. 252--264, 2001.

\bibitem{chen2005adaptive}
P.-C. Chen and A.-C. Huang, ``Adaptive sliding control of non-autonomous active
  suspension systems with time-varying loadings,'' \emph{Journal of Sound and
  Vibration}, vol. 282, no. 3-5, pp. 1119--1135, 2005.

\bibitem{gong2001neural}
J.~Gong and B.~Yao, ``Neural network adaptive robust control of nonlinear
  systems in semi-strict feedback form,'' \emph{Automatica}, vol.~37, no.~8,
  pp. 1149--1160, 2001.

\bibitem{izadbakhsh2018robust}
A.~Izadbakhsh and S.~Khorashadizadeh, ``Robust impedance control of robot
  manipulators using differential equations as universal approximator,''
  \emph{International Journal of Control}, vol.~91, no.~10, pp. 2170--2186,
  2018.

\bibitem{bai2021adaptive}
Y.~Bai, Y.~Wang, M.~Svinin, E.~Magid, and R.~Sun, ``Adaptive multi-agent
  coverage control with obstacle avoidance,'' \emph{IEEE Control Systems
  Letters}, vol.~6, pp. 944--949, 2022.

\bibitem{wang2016distributed}
G.~Wang, C.~Wang, Q.~Du, and X.~Cai, ``Distributed adaptive output consensus
  control of second-order systems containing unknown non-linear control
  gains,'' \emph{International Journal of Systems Science}, vol.~47, no.~14,
  pp. 3350--3363, 2016.

\bibitem{zirkohi2018direct}
M.~M. Zirkohi, ``Direct adaptive function approximation techniques based
  control of robot manipulators,'' \emph{Journal of Dynamic Systems,
  Measurement, and Control}, vol. 140, no.~1, 2018.

\bibitem{izadbakhsh2019fat}
A.~Izadbakhsh, P.~Kheirkhahan, and S.~Khorashadizadeh, ``{FAT}-based robust
  adaptive control of electrically driven robots in interaction with
  environment,'' \emph{Robotica}, vol.~37, no.~5, pp. 779--800, 2019.

\bibitem{huang2004adaptive}
A.-C. Huang and Y.-C. Chen, ``Adaptive sliding control for single-link
  flexible-joint robot with mismatched uncertainties,'' \emph{IEEE Transactions
  on Control Systems Technology}, vol.~12, no.~5, pp. 770--775, 2004.

\bibitem{blanchini2008set}
F.~Blanchini and S.~Miani, \emph{Set-theoretic methods in control}.\hskip 1em
  plus 0.5em minus 0.4em\relax Springer, 2008.

\bibitem{nguyen2016exponential}
Q.~Nguyen and K.~Sreenath, ``Exponential control barrier functions for
  enforcing high relative-degree safety-critical constraints,'' in
  \emph{American Control Conference (ACC)}.\hskip 1em plus 0.5em minus
  0.4em\relax IEEE, 2016, pp. 322--328.

\bibitem{xu2018constrained}
X.~Xu, ``Constrained control of input--output linearizable systems using
  control sharing barrier functions,'' \emph{Automatica}, vol.~87, pp.
  195--201, 2018.

\bibitem{efimov2013interval}
D.~Efimov, W.~Perruquetti, T.~Ra{\"\i}ssi, and A.~Zolghadri, ``On interval
  observer design for time-invariant discrete-time systems,'' in \emph{European
  Control Conference (ECC)}.\hskip 1em plus 0.5em minus 0.4em\relax IEEE, 2013,
  pp. 2651--2656.

\bibitem{moisan2007near}
M.~Moisan, O.~Bernard, and J.-L. Gouz{\'e}, ``Near optimal interval observers
  bundle for uncertain bioreactors,'' in \emph{European Control Conference
  (ECC)}.\hskip 1em plus 0.5em minus 0.4em\relax IEEE, 2007, pp. 5115--5122.

\bibitem{khalil2002nonlinear}
H.~K. Khalil, \emph{Nonlinear systems}.\hskip 1em plus 0.5em minus 0.4em\relax
  Prentice hall Upper Saddle River, NJ, 2002, vol.~3.

\bibitem{tabuada2020universal}
P.~Tabuada and B.~Gharesifard, ``Universal approximation power of deep residual
  neural networks via nonlinear control theory,'' \emph{arXiv preprint
  arXiv:2007.06007}, 2020.

\bibitem{wang2021function}
Y.~Wang, Y.~Bai, and M.~Svinin, ``Function approximation technique based
  adaptive control for chaos synchronization between different systems with
  unknown dynamics,'' \emph{International Journal of Control, Automation and
  Systems}, pp. 1--11, 2021.

\bibitem{bai2020function}
Y.~Bai, Y.~Wang, M.~Svinin, E.~Magid, and R.~Sun, ``Function approximation
  technique based immersion and invariance control for unknown nonlinear
  systems,'' \emph{IEEE Control Systems Letters}, vol.~4, no.~4, pp. 934--939,
  2020.

\bibitem{bai2021motion}
Y.~Bai, M.~Svinin, E.~Magid, and Y.~Wang, ``On motion planning and control for
  partially differentially flat systems,'' \emph{Robotica}, vol.~39, no.~4, pp.
  718--734, 2021.

\bibitem{ebeigbe2019robust}
D.~Ebeigbe, T.~Nguyen, H.~Richter, and D.~Simon, ``Robust regressor-free
  control of rigid robots using function approximations,'' \emph{IEEE
  Transactions on Control Systems Technology}, vol.~28, no.~4, pp. 1433--1446,
  2019.

\bibitem{mata2010exponential}
J.~Mata-Machuca, R.~Martinez-Guerra, and R.~Aguilar-Lopez, ``An exponential
  polynomial observer for synchronization of chaotic systems,''
  \emph{Communications in Nonlinear Science and Numerical Simulation}, vol.~15,
  no.~12, pp. 4114--4130, 2010.
\end{thebibliography}
\end{document}